\documentclass[reqno,11pt]{amsart}
\usepackage[foot]{amsaddr}
\usepackage{bbold}
\usepackage{charter}
\usepackage[margin=1in]{geometry}
\usepackage[colorlinks=true,linktoc=all,linkcolor=blue,citecolor=blue]{hyperref}

\theoremstyle{plain}
\newtheorem{theorem}{Theorem}[section]
\newtheorem*{theorem*}{Theorem}
\newtheorem{proposition}{Proposition}[theorem]
\newtheorem{lemma}[theorem]{Lemma}
\newtheorem{corollary}[theorem]{Corollary}
\theoremstyle{definition}

\newtheorem{definition}[theorem]{Definition}
\newtheorem*{definition*}{Definition}

\numberwithin{equation}{section}
\everymath{\displaystyle}
\allowdisplaybreaks

\newcommand{\C}{\mathbb{C}}
\newcommand{\D}{\mathbb{D}}
\newcommand{\N}{\mathbb{N}}
\newcommand{\Q}{\mathbb{Q}}

\newcommand{\Lip}{\mathcal{L}}
\newcommand{\LipW}{\Lip_{\textbf{w}}}
\newcommand{\Lipk}{\Lip^{(k)}}
\newcommand{\Lipko}{\Lip^{(k)}_0}
\renewcommand{\mod}[1]{\left|#1\right|}

\newcommand{\knorm}[1]{\left\|#1\right\|_{k}}

\newcommand{\Bloch}{\mathcal{B}}
\newcommand{\logBloch}{\mathcal{LB}}

\renewcommand{\a}{\alpha}

\renewcommand{\c}{\chi}

\newcommand{\e}{\varepsilon}
\newcommand{\f}{\varphi}
\newcommand{\g}{\gamma}

\renewcommand{\l}{\lambda}
\newcommand{\m}{\mu}

\renewcommand{\o}{\omega}

\newcommand{\s}{\sigma}

\renewcommand{\O}{\Omega}

\newcommand{\ben}{\begin{eqnarray}}
\newcommand{\eeqn}{\end{eqnarray}}

\author{Robert F.~Allen\textsuperscript{1}, Flavia Colonna\textsuperscript{2}, and Glenn R. Easley\textsuperscript{3}}
\address{\textsuperscript{1}Department of Mathematics and Statistics, University of Wisconsin-La Crosse}
\address{\textsuperscript{2}Department of Mathematical Sciences, George Mason University, Fairfax, VA 22030, USA}
\address{\textsuperscript{3}System Planning Corporation, Arlington, VA 22209, USA}

\email{rallen@@uwlax.edu, fcolonna@gmu.edu, geasley@sysplan.com}


\date{}
\title[Mult. Operators on Iterated Logarithmic Lipschitz Spaces]{Multiplication operators on the iterated logarithmic Lipschitz spaces\\ of a tree}

\begin{document}

\begin{abstract}
We introduce a class of iterated logarithmic Lipschitz spaces $\Lip^{(k)}$, $k\in\N$, on an infinite tree which arise naturally in the context of operator theory. We characterize boundedness and compactness of the multiplication operators on $\Lip^{(k)}$ and provide estimates on their operator norm and their essential norm. In addition, we determine the spectrum, characterize the multiplication operators that are bounded below, and prove that on such spaces there are no nontrivial isometric multiplication operators and no isometric zero divisors.
\end{abstract}

\maketitle

\section{Introduction}
Let $X$ and $Y$ be Banach spaces of complex-valued functions defined on a set $\O$, and, given a function $\psi$ defined on $\O$, define the \emph{multiplication operator with symbol $\psi$} as the linear operator $M_\psi f = \psi f,$ for all $f \in X$. One of the objectives in the study of this operator is to relate the properties of the operator to the function-theoretic properties of the symbol.  For various spaces $X$ and $Y$, much research has been done on the multiplication operator, the composition operator with symbol $\varphi$, defined as $C_\varphi f = f\circ\varphi$, and more generally the weighted composition operator $W_{\psi,\varphi} f = M_\psi C_\varphi$.  The interested reader is referred to \cite{CowenMacCluer:95} and \cite{Zhu:07}, which are valuable resources on operator theory on function spaces.

An important space of analytic functions for the study of these types of operators is the \textit{Bloch space} $\mathcal{B}$, defined to be the set of functions analytic on $\D$, known as \textit{Bloch functions}, satisfying $$\sup_{z \in \D} (1-\mod{z}^2)\mod{f'(z)} < \infty.$$ For an in-depth treatise on Bloch functions, we recommend \cite{ACP}.

Arazy \cite{Arazy:82} and Brown and Shields \cite{BrownShields:91} independently characterized the bounded multiplication operators $M_\psi$ on the Bloch space. The symbol of such operators must be bounded on $\D$ and satisfy the condition \ben \sup_{z \in \D} (1-\mod{z}^2)|\psi'(z)|\log\frac{2}{1-\mod{z}} < \infty.\label{logB}\eeqn
The set $\Bloch_{\log}$ of analytic functions $\psi$ satisfying (\ref{logB}) is a Banach space
known as the {\textit{logarithmic Bloch space}}. The space $\Bloch_{\log}$ also arises in the study of Hankel operators on the Bergman one space. The Bergman space $A^p$ on $\D$ is defined to be the set of analytic functions $f$ on $\D$ such that $$\int_{\D} \mod{f(z)}^p\;dA(z) < \infty,$$ where $A$ denotes Lebesgue area measure. The Hankel operator $H_f$ on $A^1$ is defined as $H_f g = (I-P)(\overline{f}g)$, where $I$ is the identity operator and $P$ is the standard Bergman projection from $L^1$ into $A^1$. In \cite{Attele:92}, Attele showed that the Hankel operator $H_f$ is bounded on $A^1$ if and only if $f \in \Bloch_{\log}$.

The study of operators on the logarithmic Bloch space began with the characterizations of the bounded and the compact composition operators given in \cite{Yoneda:02} by Yoneda.
In \cite{Galanopoulos:07}, Galanopoulos generalized these results by characterizing the bounded and the compact weighted composition operators on $\Bloch_{\log}$.
In \cite{HosokawaDieu:09}, Hosokawa and Dieu extended the notion of the logarithmic Bloch space by defining a family of weights on the logarithmic Bloch space, thereby obtaining a sequence $\logBloch^n$ of Banach spaces such that $\logBloch^1=\Bloch_{\log}$ and such that the symbols of the weighted composition operators on $\logBloch^n$ are precisely the bounded analytic functions in $\logBloch^{n+1}$.

Recently, the authors have begun the study of multiplication operators on the discrete environment of an infinite rooted tree, which, due to its potential theoretic properties, is widely regarded as a discrete analogue of the hyperbolic disk. In \cite{ColonnaEasley:10}, the last two authors introduced the Lipschitz space $\Lip$ on a tree, namely the set of Lipschitz functions from the tree with the edge counting metric to the Euclidean space $\C$  and showed that it can be endowed with a Banach space structure.

 The Lipschitz space is a natural discrete analogue of the Bloch space $\Bloch$, since $f\in\Bloch$ if and only if $f$ is Lipschitz as a function from $\D$ under the hyperbolic metric to the Euclidean space $\C$ (see \cite{C}, Theorem~10, or \cite{Zhu:07}, Theorem~5.6).

The characterization of the bounded multiplication operators obtained in  \cite{ColonnaEasley:10} led us to introduce in \cite{AllenColonnaEasley:10} the weighted Lipschitz space $\LipW$
 on a tree which can also be given a Banach space structure.
 As was the case with the logarithmic Bloch space, the bounded functions in the weighted Lipschitz space are precisely those that induce bounded multiplication operators on $\Lip$.
Furthermore, in the above papers, the notions of little Lipschitz space and of little weighted Lipschitz space were given. These spaces were shown to be closed separable subspaces of $\Lip$ and $\LipW$ and the bounded functions in these subspaces are precisely those that induce compact multiplication operators on the respective parent spaces.
The multiplication operators on $\Lip$ and $\LipW$ and their interplay, as well as the operators between these spaces and the space of bounded functions $L^\infty$, have been  investigated in \cite{ColonnaEasley:10}, \cite{AllenColonnaEasley:10}, \cite{ColonnaEasley:10-II} and \cite{AllenColonnaEasley:10_I}.

	In the present work, inspired by the works of Yoneda \cite{Yoneda:02}, Galanopoulos \cite{Galanopoulos:07}, and Hokosawa and Dieu \cite{HosokawaDieu:09} in the classical setting of the unit disk, we expand the focus of our study by introducing the iterated logarithmic Lipschitz spaces on a tree and extending to the multiplication operators on these spaces the results of our earlier works.

\subsection{Organization of the paper} After concluding this section by giving the preliminary definitions and notation on trees and compiling the main results on the Lipschitz and weighted Lipschitz spaces in our earlier works, in Section~\ref{Section:ILS}, we define the sequence of iterated logarithmic Lipschitz spaces $\{\Lip^{(k)}\}_{k\in\N}$ and prove that it can be endowed with a Banach space structure.
For $k\in\N$, we also define the iterated logarithmic little Lipschitz space $\Lipko$ and prove that it is a closed separable subspace of $\Lipk$.

In Section~\ref{Bdd}, we characterize the bounded multiplication operators on $\Lipk$ and $\Lipko$ and give operator norm estimates.

In Section~\ref{Section:Spectrum}, we show that the point spectrum is nonempty and, in fact, it is a dense subset of the spectrum. We also show that the spectrum and the approximate point spectrum are equal to the closure of the range of the symbol.  We deduce a characterization of the bounded multiplications operators that are bounded below.

In Section~\ref{isometries}, we show that among the multiplication operators on $\Lipk$ there are no nontrivial isometries and no isometric zero divisors, thereby extending the results obtained in \cite{AllenColonnaEasley:10} in the case $k=1$.

In Section~\ref{Section:Compact}, we provide a characterization of the compact multiplication operators on $\Lipk$ and $\Lipko$. Finally, in Section~\ref{ess_norm}, we give essential norm estimates.

\subsection{Preliminaries and earlier results}
By a \emph{tree} $T$ we mean a locally finite, connected, and simply-connected graph, which, as a set, we identify with the collection of its vertices.  By a \emph{function on a tree} we mean a complex-valued function on the set of its vertices.

Two vertices $v$ and $w$ are called \emph{neighbors} if there is an edge $[v,w]$
connecting them, and we use the notation $v\sim w$. A vertex is called \emph{terminal} if it has a unique neighbor. A \emph{path} is a finite or infinite sequence of vertices $[v_0,v_1,\dots]$ such that $v_k\sim v_{k+1}$ and  $v_{k-1}\ne v_{k+1}$, for all $k$. Given a tree $T$ rooted at $o$ and a vertex $v\in T$, a vertex $w$ is called \emph{descendant} of $v$  if $v$ lies in the unique path from $o$ to $w$. The vertex $v$ is then called an \emph{ancestor} of $w$.  The vertex $v$ is called a \emph{child} of $v^-$.

For $v\in T$, the set $S_v$ consisting of $v$ and all its descendants is called the \emph{sector} determined by $v$. Define the \emph{length} of a finite path
$[v=v_0,v_1,\dots,w=v_n]$ (with $v_k\sim v_{k+1}$ for $k=0,\dots, n-1$) to be the number $n$ of edges connecting $v$ to $w$. The \emph{distance}, $d(v,w)$, between vertices $v$ and $w$ is the length of the unique path connecting $v$ to $w$. Fixing $o$ as the root of the tree, we define the \emph{length} of a vertex $v$, by $|v|=d(o,v)$.

In this paper, we shall assume the tree $T$ to be without terminal vertices (and hence infinite), and rooted at a vertex $o$.  We denote by $L^\infty$ the space of the bounded functions $f$ on the tree equipped with the supremum norm $\|f\|_\infty=\sup\limits_{v\in T}|f(v)|.$

The Lipschitz space $\Lip$ on a tree T introduced in \cite{ColonnaEasley:10} can be more simply described as the space of functions $f$ on $T$ whose ``derivative" defined by $Df(v) = \mod{f(v)-f(v^-)}$ is bounded on $T^* = T\setminus\{o\}$. The norm $\|f\|_\Lip = \mod{f(o)} + \|Df\|_\infty$ renders $\Lip$ a Banach space. The bounded and compact multiplication operators on $\Lip$ were characterized as follows.

\begin{theorem} [\cite{ColonnaEasley:10}, Theorems 3.6 and 7.2] \label{CE} For a function $\psi$  on a tree $T$,
\begin{enumerate}
\item $M_\psi$ is bounded on $\Lip$ if and only if $\psi \in L^\infty$ and $\displaystyle\sup_{v \in T^*} \mod{v}D\psi(v) < \infty.$
\item $M_\psi$ is compact on $\Lip$ if and only if $\displaystyle\lim_{\mod{v}\to\infty} \psi(v) = 0$ and $\displaystyle\lim_{\mod{v}\to\infty} \mod{v}D\psi(v) = 0.$
\end{enumerate}
\end{theorem}

In \cite{OhnoZhao:01}, Ohno and Zhao characterized the compact weighted composition operators on the Bloch space. As a consequence they deduced that the only compact multiplication operator on the Bloch space is the operator whose symbol is identically $0$. The same holds true for the iterated logarithmic Bloch spaces \cite{HosokawaDieu:09}. By contrast, the above result shows that the Lipschitz space is richer for the study of multiplication operators, since the set of compact multiplication operators on $\Lip$ contains, not only the operators whose symbol is a function on $T$ with finite support, but also nonvanishing functions such as $\psi(v) = \frac{1}{|v|^p}$ for $v \neq o$, where $p \geq \frac{1}{2}$ and $\psi(o)$ is defined arbitrarily.

In \cite{AllenColonnaEasley:10}, we introduced the weighted Lipschitz space $\LipW$ on a tree T to be the space of functions on $T$ for which $\sup_{v \in T^*} |v|Df(v) < \infty$ and proved that $\LipW$ is a Banach space under the norm
$\|f\|_{\textbf{w}}=|f(0)|+\sup_{v \in T^*} \mod{v}D\psi(v).$
Furthermore, we defined the \textit{little weighted Lipschitz space} to be the set $\Lip_{\textbf{w},0}$ of functions $f\in\LipW$ such that
$\lim\limits_{|v|\to\infty}\mod{v}D\psi(v)=0,$
and proved that $\Lip_{\textbf{w},0}$ is a closed separable subspace of $\LipW$.

 The study of the multiplication operators on $\LipW$ led to the following result.

\begin{theorem}[\cite{AllenColonnaEasley:10}, Theorems 4.1 and 7.2]\label{ACE} Given a function $\psi$ on a tree $T$,\smallskip

\noindent \hskip3pt {\rm{(1)}}\ $M_\psi$ is bounded on $\Lip_{\rm{\bf{w}}}$ if and only if $\psi \in L^\infty$ and $\sup\limits_{v \in T^*} |v|\log|v|D\psi(v) < \infty.$

\noindent \hskip3pt {\rm{(2)}}  $M_\psi$ is compact on $\Lip_{\rm{\bf{w}}}$ if and only if\hskip-2pt $\lim\limits_{|v|\to\infty}\hskip-2pt \psi(v)\hskip-2pt~=~\hskip-2pt 0$ and $\lim\limits_{|v|\to\infty}\hskip-3pt |v|\log|v|D\psi(v)\hskip-3pt~=~\hskip-2pt0.$
\end{theorem}

The logarithmic term present in the above characterizations of boundedness and compactness of $M_\psi$ suggested the existence of a decreasing chain of spaces and corresponding subspaces such that the elements of a space (respectively, subspace) of a given generation could be used to characterize the bounded (respectively, compact) multiplication operators on the space relative to the parent generation.

\section{The iterated logarithmic Lipschitz spaces}\label{Section:ILS}

Let $T$ be a tree rooted at $o$ and fix $k\in\N$.

\begin{definition}\label{Lipit} The {\textit{iterated logarithmic Lipschitz space}} $\Lip^{(k)}$ is the set of functions $f$ on $T$ satisfying the condition $$\sup_{v\in T^*}|v|\prod_{j=0}^{k-1}\ell_{j}(|v|)Df(v)<\infty,$$
where for $x\ge 1$ the sequence $\ell_j(x)$ is defined recursively by
\ben \ell_{j}(x)=\begin{cases}1&\quad\hbox{ for }j=0,\\
1+\ln x&\quad\hbox{ for }j=1,\\1+\ln \ell_{j-1}(x)&\quad\hbox{ for }j\ge 2.\end{cases}\label{indcond}\eeqn
\end{definition}

For $f\in \Lipk$, define
$$\|f\|_k=|f(o)|+\sup_{v\in T^*}|v|\prod_{j=0}^{k-1}\ell_j(|v|)Df(v). $$

Notice that $\Lip^{(1)}$ is precisely the weighted Lipschitz space $\LipW$ and $\|f\|_1=\|f\|_{\textbf{w}}$.

\begin{proposition}\label{modulus_est} If \rm{$f\in \Lip^{(k)}$} and $v\in T^*$, then
\rm{\ben |f(v)|\le \ell_k(|v|)\|f\|_k.\label{modulusest}\eeqn}
\end{proposition}

For the proof we need the following results.

\begin{lemma}\label{easy} For all $k\in\N$, the sequence $\{\a_k(n)\}_{n=2}^\infty$ defined by
$\a_k(n)=\frac{\ell_k(n)}{\ell_k(n-1)}$ is decreasing and
\ben \displaystyle\lim\limits_{n\to\infty}\frac{\ell_k(n)}{\ell_k(n-1)}=1.\label{limit1}\eeqn
\end{lemma}

\noindent\textit{Proof.} We first show that $\a_k(n)$ is decreasing. As an immediate consequence of (\ref{indcond}), observe that for each $x>1$,
\ben \ell_{k}'(x)=\frac1{x\prod_{j=0}^{k-1}\ell_j(x)}. \label{derivative}\eeqn For $\a_k$ to be decreasing, the derivative must be negative. Thus, it suffices to show that $\ell_k'(n)\ell_k(n-1)<\ell_k'(n-1)\ell_k(n),$ that is,
$$\frac{\ell_k(n-1)}{n\prod_{j=0}^{k-1}\ell_j(n)}<\frac{\ell_k(n)}{(n-1)\prod_{j=0}^{k-1}\ell_j(n-1)}.$$
This is equivalent to $$(n-1)\prod_{j=0}^k\ell_j(n-1)< n\prod_{j=0}^k\ell_j(n),$$
which is clearly satisfied.

To show (\ref{limit1}), we argue by induction on $k$. For $k=1$ the result is immediate. Assume the result true for all $j\le k-1$ (where $k\ge 2$). Then by l'H\^opital's rule, we have
\ben \lim_{n\to\infty}\frac{\ell_k(n)}{\ell_k(n-1)}&=&\lim_{n\to\infty}\frac{1+\ln \ell_{k-1}(n)}{1+\ln \ell_{k-1}(n-1)}=\lim_{n\to\infty}\frac{\ell_{k-1}(n-1)\ell_{k-1}'(n)}{\ell_{k-1}(n)\ell_{k-1}'(n-1)}.\label{hop}\eeqn
Thus, (\ref{hop}), (\ref{derivative}), and the induction hypothesis yield
$$\lim_{n\to\infty}\frac{\ell_k(n)}{\ell_k(n-1)}=\lim_{n\to\infty}\left(\frac{n-1}{n}\right)\prod_{j=0}^{k-1}\frac{\ell_{j}(n-1)}{\ell_{j}(n)}=1.\hskip20pt\square$$

\begin{lemma}\label{noteasy} For $k\in\N$ the sequence $\{\f_{k,n}\}_{n=2}^\infty$ defined by
$$\f_{k,n}=n\prod_{j=0}^{k-1}\ell_{j}(n)\left[\ell_k(n)-\ell_k(n-1)\right]-1$$ is positive, decreasing, and $\lim\limits_{n\to\infty}\f_{k,n}=0$.
\end{lemma}

\begin{proof} To prove that $\{\f_{k,n}\}$ is decreasing, we shall argue by induction on $k$. For $k=1$, $\f_{k,n}=n(\ln(n)-\ln(n-1))-1$ which is decreasing. Assume $\f_{k,n}$ is decreasing for some $k\in\N$. Observe that
\ben \frac{\f_{k+1,n}+1}{\f_{k,n}+1}&=&\frac{n\prod_{j=0}^{k}\ell_j(n)[\ell_{k+1}(n)-\ell_{k+1}(n-1)]}{n\prod_{j=0}^{k-1}\ell_j(n)[\ell_{k}(n)-\ell_{k}(n-1)]}\nonumber\\
&=&\ell_{k}(n)\left[\frac{\ln \ell_{k}(n)-\ln \ell_{k}(n-1)}{\ell_{k}(n)-\ell_{k}(n-1)}\right]\nonumber\\
&=&\frac{\ell_{k}(n)}{\ell_{k}(n-1)}\left[\frac{\ln\left(\frac{\ell_{k}(n)}{\ell_{k}(n-1)}\right)}{\frac{\ell_{k}(n)}{\ell_{k}(n-1)}-1}\right]=\g_{k,n}.\label{decr}\eeqn
The function $\g_{k,n}$ can be expressed as
$$\g_{k,n}=\a_k(n)\left[\frac{\ln(\a_k(n))}{\a_k(n)-1}\right],$$
which is increasing as a function of $\a_k(n)$. However, by Lemma~\ref{easy}, $\a_k(n)$ decreases to 1 as $n$ increases. Thus, $\g_{k,n}$ is decreasing in $n$.
Therefore, by (\ref{decr}) and by the inductive hypothesis, we see that
$\f_{k+1,n}+1=(\f_{k,n}+1)\g_{k,n}$, which is decreasing as a product of two positive decreasing functions. Hence $\f_{k+1,n}$ is decreasing in $n$ for each $k\in\N$.

We next show that
\ben \displaystyle\lim\limits_{n\to\infty}\f_{k,n}=0,\label{limit0}\eeqn arguing by induction on $k$. From this it will follow that the sequence is positive.  For $k=1$, we have
$$\lim_{n\to\infty}\f_{1,n}=\lim_{n\to\infty}n\left[1+\ln n-(1+\ln(n-1)\right]-1=\lim_{n\to\infty}n\ln\left[\frac{n}{n-1}\right]-1=0.$$
Assume (\ref{limit0}) holds for some $k\ge 1$. Then, by (\ref{limit1}) and noting that
$\displaystyle\lim\limits_{\a\to 1}\frac{\ln \a}{\a-1}=1$, we obtain $\lim\limits_{n\to\infty}\g_{k,n}=1$, whence
$$\lim_{n\to\infty}\frac{\f_{k+1,n}+1}{\f_{k,n}+1}=1.$$
From the inductive hypothesis we deduce that $\lim\limits_{n\to\infty}\f_{k+1,n}=0.$
\end{proof}

\noindent{\textit{Proof of Proposition~\ref{modulus_est}.}} Let us first assume $f(o)=0$ and argue by induction on $|v|$. Since $\ell_k(1)=1$ for each $k\in\N$, for $|v|=1$, we have $$|f(v)|=\ell_k(|v|)Df(v)\le \|f\|_k=\ell_k(|v|)\|f\|_k.$$
Let $n\in\N$, $n\ge 2$, and assume $|f(w)|\le \ell_k(|w|)\|f\|_k$ whenever $w$ is a vertex such that $1\le |w|<n$. Let $v$ be a vertex of length $n$. Then, by Lemma~\ref{noteasy} we obtain
\ben |f(v)|&\le &|f(v^-)|+|f(v)-f(v^-)|\nonumber\\&\le& \left[\ell_k(|v|-1)+\frac1{|v|\prod_{j=0}^{k-1}\ell_j(|v|)}\right]\|f\|_k
\le \ell_k(|v|)\|f\|_k.\nonumber\eeqn

On the other hand, if $f(o)\ne 0$, letting $g(v)=f(v)-f(o)$ for $v\in T$, we obtain $|g(v)|\le \ell_k(|v|)\|g\|_k$ for $v\in T^*$. Since $\|f\|_k=|f(o)|+\|g\|_k$ and $\ell_k(|v|)\ge 1$, we deduce  \ben |f(v)|\le |f(o)|+|g(v)|\le |f(o)|+\ell_k(|v|)\|g\|_k\le \ell_k(|v|)\|f\|_k,\nonumber\eeqn completing the proof.\hskip260pt$\square$

\begin{theorem}\label{Banach} The iterated logarithmic Lipschitz space $\Lip^{(k)}$ is a Banach space under the norm $\|\cdot\|_k$.
\end{theorem}

\begin{proof} It is immediate to see that $\Lipk$ is a vector space and that $f\mapsto \knorm{f}$ is a semi-norm. It is also evident that the norm of the function identically 0 is 0. Conversely, assume $\knorm{f}=0$. Then $Df$ is identically 0. Thus, $f$ is a constant and since $f(o)=0$, $f$ is identically 0.

To prove that $\Lipk$ is a Banach space, let $\{f_n\}$ be Cauchy in $\Lipk$. For $n,m\in \N$, since $|f_n(o)-f_m(o)|\le \knorm{f_n-f_m}$, and by Proposition~\ref{modulus_est}, for $v\in T^*$,
\ben |f_n(v)-f_m(v)|\le \ell_k(|v|)\knorm{f_n-f_m},\nonumber\eeqn
the sequence $\{f_n(v)\}$ is Cauchy for each $v\in T$. Hence it converges pointwise to some function $f$.
We now show that $f\in\Lipk$.

Fix $v\in T^*$ and $\e>0$. Then there exists $N\in\N$ such that for all $n\ge N$
$$|f_n(v)-f(v)|<\displaystyle\frac{\e}{2|v|\prod_{j=0}^{k-1}\ell_j(|v|)}\ \hbox{ and }\ |f_n(v^-)-f(v^-)|<\displaystyle\frac{\e}{2|v|\prod_{j=0}^{k-1}\ell_j(|v|)}.$$ Thus, for $n\ge N$, we have
\ben\label{Eq1}|v|\prod_{j=0}^{k-1}\ell_j(|v|)Df(v) &\le& |v|\prod_{j=0}^{k-1}\ell_j(|v|)|f(v)-f_n(v)|+|v|\prod_{j=0}^{k-1}\ell_j(|v|)Df_n(v)\\&\ &\quad +|v|\prod_{j=0}^{k-1}\ell_j(|v|)|f_n(v^-)-f(v^-)|\nonumber\\
&\le& \e+\knorm{f_n}.\nonumber\eeqn
Since $\{f_n\}$ is Cauchy in $\Lipk$, and hence bounded, $\{\knorm{f_n}\}$ is uniformly bounded by some constant $C$, and so letting $\e\to 0$, (\ref{Eq1}) yields $|v|\prod_{j=0}^{k-1}\ell_{j}(|v|)Df(v)\le C.$ Therefore $f\in\Lipk$.

To conclude the proof of the completeness, we need to show that $\{f_n\}$ converges to $f$ in norm as $n\to\infty$.  Since $f_n(o)\to f(o)$, it suffices to show that $$\sup\limits_{v\in T^*}|v|\prod_{j=0}^{k-1}\ell_{j}(|v|)D(f_n-f)(v)\to 0$$ as $n\to\infty$. Arguing by contradiction, suppose there exist $\e>0$ and a subsequence $\{f_{n_h}\}_{h\in\N}$
 such that $\displaystyle\sup_{v\in T^*}|v|\prod_{j=0}^{k-1}\ell_{j}(|v|)D(f_{n_h}-f)(v)>\e$ for all $h\in\N$. Then for each $h\in\N$, we may pick two neighboring vertices $v_{n_h}$ and $w_{n_h}$, with $v_{n_h}$ child of $w_{n_h}$, such that $$|v_{n_h}|\prod_{j=0}^{k-1}\ell_{j}(|v_{n_h}|)\left|f_{n_h}(v_{n_h})-f(v_{n_h})-(f_{n_h}(w_{n_h})-f(w_{n_h}))\right|\ge \e.$$  Since $\{f_{n_h}\}$ is Cauchy in $\Lipk$, there exists a positive integer $h_0$ such that for each $h,h'\ge h_0$, and $v\in T^*$, we have
$$|v|\prod_{j=0}^{k-1}\ell_{j}(|v|)\left|f_{n_h}(v)-f_{n_{h'}}(v)-(f_{n_{h}}(v^-)-f_{n_{h'}}(v^-))\right|\le \knorm{f_{n_h}-f_{n_{h'}}}<\frac{\e}{2}.$$ In particular, for all $h\ge h_0$, we have \small \ben\label{firstest} \ \ \  
|v_{n_{h_0}}|\prod_{j=0}^{k-1}\ell_{j}(|v_{n_{h_0}}|)\left|f_{n_{h_0}}(v_{n_{h_0}})-f_{n_{h}}(v_{n_{h_0}})-(f_{n_{h_0}}(w_{n_{h_0}})-f_{n_{h}}(w_{n_{h_0}}))\right|<\frac{\e}{2}.\eeqn \normalsize On the other hand, by the pointwise convergence of $f_{n_h}$ to $f$, for all integers $h$ sufficiently large
\small \ben\label{sndest} \ \ \ \ \ \left|f_{n_h}(v_{n_{h_0}})-f(v_{n_{h_0}})-(f_{n_{h}}(w_{n_{h_0}})-f(w_{n_{h_0}}))\right|<\frac{\e}{2|v_{n_{h_0}}|\prod_{j=0}^{k-1}\ell_{j}(|v_{n_{h_0}}|)}.\eeqn\normalsize
Applying the triangle inequality, from (\ref{firstest}) and (\ref{sndest}) we obtain
$$|v_{n_{h_0}}|\prod_{j=0}^{k-1}\ell_{j}(|v_{n_{h_0}}|)\left|f_{n_{h_0}}(v_{n_{h_0}})-f(v_{n_{h_0}})-(f_{n_{h_0}}(w_{n_{h_0}})-f(w_{n_{h_0}}))\right|<\e,$$ contradicting the choice of $v_{n_{h_0}}$ and $w_{n_{h_0}}$. Therefore $\Lipk$ is a Banach space.\end{proof}

A Banach space $X$ of complex-valued functions on a set $\Omega$ is said to be a \emph{functional Banach space} if for each $\o\in \Omega$, the point evaluation functional $e_\o(f) = f(\o),\, f \in X,$ is bounded.

\begin{lemma}\label{drs}{\rm{(\cite{DurenRombergShields:69}, Lemma 11)}} Let $X$ be a functional Banach space on the set $\Omega$ and let $\psi$ be a complex-valued function on $\Omega$ such that $M_\psi$ maps $X$ into itself. Then $M_\psi$ is bounded on $X$ and $|\psi(\o)|\le \|M_\psi\|$ for all $\o\in\O$. In particular, $\psi$ is bounded.\end{lemma}

\begin{corollary}\label{funct_Banach} The set \rm{$\Lipk$} is a functional Banach space. If $M_\psi$ is a multiplication operator on \rm{$\Lipk$}, then $M_\psi$ is bounded, its symbol $\psi$ is bounded and $\|\psi\|_\infty\le \|M_\psi\|$.
\end{corollary}

\begin{proof} For $f\in\Lipk$, $|f(o)|\le \knorm{f}$. Furthermore, fixing $v\in T^*$, inequality (\ref{modulusest}) shows that the point evaluation functional $e_v(f)=f(v)$ is bounded. Thus, $\Lipk$ is a functional Banach space. The second statement follows immediately from Lemma~\ref{drs}.\end{proof}

\begin{definition}\label{littleLip} The {\it{iterated logarithmic little Lipschitz space}} is the subspace $\Lipko$ of $\Lipk$ consisting of the functions $f$ such that $$\lim_{|v|\to\infty}|v|\prod_{j=0}^{k-1}\ell_j(|v|)Df(v)=0.$$\end{definition}

\begin{proposition}\label{lm0} If \rm{$f\in\Lipko$}, then $\displaystyle\lim\limits_{|v|\to\infty}\frac{f(v)}{\ell_k(|v|)}=0.$
\end{proposition}

\begin{proof} If $f$ is constant, then $f$ is identically 0 and the result holds trivially. So assume $f$ is nonconstant and set $L=\sup\limits_{v\in T^*}|v|\prod_{j=0}^{k-1}\ell_j(|v|)Df(v)>0$. Fix $\e\in (0,L)$. Since $f\in\Lipko$, there exists $N\in\N$ such that $|v|\prod_{j=0}^{k-1}\ell_{j}(|v|)Df(v)<\e$, for all $v\in T$, with $|v|\ge N$. For $|w|=N$ and $v$ a descendant of $w$, let $u_0=w,u_1,\dots,u_{|v|-N}=v$ be the vertices in the path from $w$ to $v$, where ${u_h}^-=u_{h-1}$, $h=1,\dots,|v|-N$. By Lemma~\ref{noteasy}, we have
$$\ell_k(N)+\frac1{|u_1|\prod_{j=0}^{k-1}\ell_j(|u_1|)}\le \ell_k(|u_1|).$$
Thus, arguing inductively on $h$, we obtain
\ben \ell_k(N)+\sum_{h=1}^{|v|-N}\frac{1}{|u_h|\prod_{j=0}^{k-1}\ell_j(|u_h|)}\le \ell_k(|v|).\label{use}\eeqn
By the triangle inequality, Proposition~\ref{modulus_est}, and (\ref{use}), we have
\ben |f(v)|&\le &|f(w)|+\sum_{h=1}^{|v|-N}|f(u_h)-f(u_{h-1})|\nonumber\\
&\le &\ell_k(N)\knorm{f}+\e\sum_{h=1}^{|v|-N}\frac{1}{|u_h|\prod_{j=0}^{k-1}\ell_j(|u_h|)}\nonumber\\
&=&\ell_k(N)\left(\|f\|_k-\e\right)+\e\left[\ell_k(N)+\sum_{h=1}^{|v|-N}\frac1{{|u_h|\prod_{j=0}^{k-1}\ell_j(|u_h|)}}\right]\nonumber\\
&\le & \ell_k(N)\left(\|f\|_k-\e\right)+\e\ell_k(|v|).\nonumber\eeqn
Therefore, for all vertices $v$ of length greater than $N$ we obtain
$$\frac{|f(v)|}{\ell_k(|v|)}\le\frac{\ell_k(N)\left(\|f\|_k-\e\right)}{\ell_k(|v|)}+\e.$$
Hence $\lim\limits_{|v|\to\infty}\displaystyle\frac{|f(v)|}{\ell_k(|v|)}\le \e$. Letting $\e\to 0$, we obtain the result.\end{proof}

Let $\chi_A$ denote the characteristic function of the set $A$ and for $v\in T$ adopt the notation $\chi_v$ for the function $\chi_{\{v\}}$. Furthermore, for $v\in T$, let $p_v=\c_{S_v}$, where we recall that $S_v$ is the set consisting of $v$ and all its descendants.

\begin{proposition}\label{dense} The set $$\mathcal{P}=\left\{\sum\limits_{k=1}^N a_kp_{v_k}: N\in\N, a_k\in\C, v_k\in T, 1\le k\le N\right\}$$ is dense in $\Lipko$ with respect to the norm topology. In fact, $\Lipko$ is a closed separable subspace of $\Lipk$.
\end{proposition}

\begin{proof} Fix $v\in T$ and observe that $Dp_v=\c_v$, so that for $w\in T^*$, we have
$$|w|\prod_{j=0}^{k-1}\ell_j(|w|)Dp_v(w)=\begin{cases} 0&\quad\hbox{ if }w\ne v,\\
|v|\prod_{j=0}^{k-1}\ell_j(|v|)& \quad\hbox{ if }w= v.\end{cases}$$
Thus, $p_v\in\Lipko$, and so $\mathcal{P}$ is a subspace of $\Lipko$.

Let $f\in\Lipko$ and for $n\in\N$, define
$$f_n(v)=\begin{cases} f(v)& \quad\hbox{ if }|v|\le n,\\
f(v_n)&\quad\hbox{ if }|v|>n,\end{cases}$$
where $v_n$ is the ancestor of $v$ of length $n$. Observe that for $v\in T^*$,
\ben \chi_v=p_v-\sum_{w\in [v^+]}p_w,\label{chi_v}\eeqn
where $[v^+]$ denotes the set of children of $v$.
Therefore, for $n\in\N$, we have
\ben f_n&=&\sum_{|v|<n}f(v)\c_v +\sum_{|v|=n}f(v)p_v\nonumber\\
&=&\sum_{|v|<n}f(v)\left(p_v-\sum_{w\in [v^+]}p_w\right) +\sum_{|v|=n}f(v)p_v\nonumber\\
&=&\sum_{|v|\le n}f(v)p_v-\sum_{|v|<n}f(v)\sum_{w\in [v^+]}p_w.\nonumber\eeqn
Thus, $f_n$ is a finite linear combination of the functions $p_v$ and
$$\knorm{f_n-f}=\sup_{|v|>n}|v|\prod_{j=0}^{k-1}\ell_j(|v|)Df(v)\to 0$$
as $n\to\infty$.

Next observe that $\Q[i]=\{z\in\C: {\rm{Re}}\, z, {\rm{Im}}\, z\in\Q\}$ is dense in $\C$, and $T$ is countable. Thus, the subset of $\mathcal{P}$ consisting of the finite linear combinations of the functions $p_v$ with coefficients in $\Q[i]$ is countable and dense in $\Lipko$. Therefore, $\Lipko$ is a closed separable subspace of $\Lipk$.\end{proof}

The following result will be used in Section~\ref{ess_norm} to derive estimates on the essential norm of the multiplication operators on $\Lipk$.

\begin{proposition}\label{weakconv} Let $\{f_n\}$ be a sequence of functions in \rm{$\Lipko$} converging to $0$ pointwise in $T$ and such that \rm{$\knorm{f_n}$} is bounded. Then $f_n\to 0$ weakly in \rm{$\Lipko$}.
\end{proposition}

\begin{proof} First suppose $f_n(o)=0$ for all $n\in\N$, so that
$$\|f_n\|_k=\sup\limits_{v\in T^*}|v|\prod_{j=0}^{k-1}\ell_j(|v|)Df_n(v).$$ Then, letting $\m(v)=|v|\prod_{j=0}^{k-1}\ell_j(|v|)$ for $v\in T$, the sequence $\{\m Df_n\}$ converges to 0 pointwise. Observe that the subspace of $\Lipko$ whose elements send $o$ to 0 is isomorphic to the space $c_0$, consisting of the sequences indexed by $T$ which vanish at infinity, under the supremum norm via the correspondence $f\mapsto \m Df$. The space $c_0$ has dual isomorphic to the space $\ell^1$ of absolutely summable sequences (e.g. \cite{Conway:07}) via the correspondence $g\in \ell^1\mapsto \widetilde{g}\in c_0^*$, where for $f\in c_0$, $\widetilde{g}(f)=\sum_{v\in T}f(v)g(v).$  Thus, under the identification of $\Lipko$ with $c_0$, if $f_n\in c_0$ converges pointwise to 0 and is bounded in $c_0$, then for any $g\in \ell^1$, we have
\ben |\widetilde{g}(f_n)|=\left|\sum_{v\in T}f_n(v)g(v)\right|\le \sum_{v\in T}|f_n(v)||g(v)|.\label{estgfn}\eeqn
Let $c=\sup\limits_{n\in \N,v\in T}|f_n(v)|$. Fixing any positive integer $N$, we may split the sum on the right-hand side of (\ref{estgfn}) into the two sums $$S_1(n,N)=\sum_{|v|\le N}|f_n(v)||g(v)|\ \hbox{ and }\ S_2(n,N)=\sum_{|v|> N}|f_n(v)||g(v)|.$$ Since $f_n\to 0$ uniformly on the set $\{v\in T: |v|\le N\}$, we see that
$$S_1(n,N)\le \max_{|v|\le N}|f_n(v)|\|g\|_1\to 0,\hbox{ as }n\to\infty.$$
On the other hand, since $g\in \ell^1$, the tail of the series $\displaystyle\sum_{v\in T}|g(v)|$ approaches 0.
Therefore, from the inequalities
$$\lim_{n\to\infty}|\widetilde{g}(f_n)|\le \lim_{n\to\infty}S_1(n,N)+\sup_{n\in\N}S_2(n,N)\le c\sum_{|v|>N}|g(v)|,$$ letting $N\to \infty$, we deduce that $\displaystyle\lim_{n\to\infty}\widetilde{g}(f_n)=0$.

Hence, if $f_n(o)=0$, then $f_n$ converges to 0 weakly. In the general case, define $F_n=f_n-f_n(o)$. By the previous part, $F_n\to 0$ weakly. Since $f_n(o)\to 0$, we conclude that $f_n\to 0$ weakly as well.\end{proof}

\section{Boundedness and operator norm}\label{Bdd}

In this section, we characterize the symbol of the bounded multiplication operators on the iterated Lipschitz spaces and obtain estimates on the operator norm, which reduce to the  norm estimates in \cite{AllenColonnaEasley:10} in the case $k=1$.

\begin{theorem}\label{boundedness} For a function $\psi$ on $T$, the following statements are equivalent.
\begin{enumerate}
\item[\rm{(a)}] $M_\psi$ is bounded on $\Lipk$.
\item[\rm{(b)}] $M_\psi$ is bounded on $\Lipko$.
\item[\rm{(c)}] $\psi\in L^\infty\cap \Lip^{(k+1)}$.
\end{enumerate}
Furthermore, under the above conditions, the following estimates hold:
$$\max\{\|\psi\|_\infty,\|\psi\|_k\} \leq \|M_\psi\| \leq \|\psi\|_\infty + \sup_{v \in T^*} |v|\prod_{j=1}^k \ell_j(|v|)D\psi(v).$$
\end{theorem}

\begin{proof} (a)$\Longrightarrow$(c):  Assume $M_\psi$ is bounded on $\Lipk$. The boundedness of $\psi$ is an immediate consequence of Corollary~\ref{funct_Banach}. To prove that $\psi\in\Lip^{(k+1)}$, for $v\in T$, define $$f(v)=\begin{cases}\ \ 0 &\quad\hbox{ if }v=o,\\
\ell_k(|v|)&\quad\hbox{ if }v\in T^*.\end{cases}$$  Then, for $|v|=1$, we have $Df(v)=1$, while for $|v|\ge 2$, by Lemma~\ref{noteasy}, $$|v|\prod_{j=0}^{k-1}\ell_{j}(|v|)Df(v)\le 2\prod_{j=0}^{k}\ell_{j}(2),$$
so $f\in\Lipk$. Furthermore, for $v\in T^*$ we have
\ben D\psi(v)f(v)&\le & |\psi(v)f(v)-\psi(v^-)f(v^-)|+|\psi(v^-)f(v^-)-\psi(v^-)f(v)|\nonumber\\
&= & D(\psi f)(v)+|\psi(v^-)|Df(v).\label{prodrule2}\eeqn
By the boundedness of $M_\psi$, for $v\in T^*$, we obtain
\ben &\,&\hskip-40pt \sup_{v\in T^*}|v|\prod_{j=0}^{k}\ell_j(|v|)D\psi(v)=|v|\prod_{j=0}^{k-1}\ell_j(|v|)D\psi(v)f(v)\nonumber\\&\,&\hskip66pt \le |v|\prod_{j=0}^{k-1}\ell_j(|v|)D(\psi f)(v)
+|\psi(v^-)||v|\prod_{j=0}^{k-1}\ell_j(|v|)Df(v)\nonumber\\
&\,& \hskip66pt \le \knorm{M_\psi f}+\|\psi\|_\infty\knorm{f}.\nonumber\eeqn
Therefore $\psi\in \Lip^{(k+1)}$.

(c)$\Longrightarrow$(a):  Assume $\psi\in L^\infty\cap\Lip^{(k+1)}$. For $f\in\Lipk$ and $v\in T^*$, we have
\ben D(\psi f)(v)&\le & |\psi(v)f(v)-\psi(v^-)f(v)|+|\psi(v^-)f(v)-\psi(v^-)f(v^-)|\nonumber\\
&=& D\psi(v)|f(v)|+|\psi(v^-)|Df(v).\label{prodrule}\eeqn
Thus, using (\ref{prodrule}) and Proposition~\ref{modulus_est}, we obtain
\ben \|M_\psi f\|_k
&\le & |\psi(o)f(o)|+|v|\prod_{j=0}^{k-1}\ell_j(|v|)D\psi(v)|f(v)|\nonumber\\
&\ &\quad +|\psi(v^-)||v|\prod_{j=0}^{k-1}\ell_j(|v|)Df(v)\nonumber\\
&\le &\left[|v|\prod_{j=0}^{k}\ell_j(|v|)D\psi(v)+\|\psi\|_\infty\right]\knorm{f}\label{ester}\\
&\le &\left(\|\psi\|_{k+1}+\|\psi\|_\infty\right)\knorm{f}\nonumber\eeqn
Therefore $\psi f\in\Lipk$ and from (\ref{ester}) we obtain the upper estimate on the operator norm.

(b)$\Longrightarrow$(c):  Suppose $M_\psi$ is bounded on $\Lipko$. Again by Corollary~\ref{funct_Banach}, we see that $\psi\in L^\infty$. To show that $\psi\in \Lip^{(k+1)}$, for $0<p<1$, define the function $f_p$ on $T$ by
$$f_p(v)=\begin{cases} 0 &\quad\hbox{ if }v=o,\\(\ell_k(|v|))^p&\quad\hbox{ if }v\ne o.\end{cases}$$
Since $\displaystyle\lim_{n\to\infty}\frac{\left(\ell_k(n)\right)^p-\left(\ell_k(n-1)\right)^p}{\ell_k(n)-\ell_k(n-1)}=0,$ and by Lemma~\ref{noteasy},
$$\lim_{|v|\to \infty}|v|\prod_{j=0}^{k-1}\ell_j(|v|)\left(\ell_k(|v|)-\ell_k(|v|-1)\right)=1,$$ we deduce that $|v|\prod_{j=0}^{k-1}\ell_j(|v|)Df_p(v)\to 0$ as $|v|\to\infty$. Thus $f_p\in\Lipko$. Since for $0<p<1$, the function $x\mapsto x-x^p$ is increasing for $x\ge 1$, it follows that for $|v|\ge 2$, $Df_p(v)\le \ell_k(|v|)-\ell_k(|v|-1)$, so by Lemma~\ref{noteasy}, we have
$$ |v|\prod_{j=0}^{k-1}\ell_j(|v|)Df_p(v)\le |v|\prod_{j=0}^{k-1}\ell_j(|v|)(\ell_k(|v|)-\ell_k(|v|-1))\le 2\prod_{j=0}^{k}\ell_{j}(2).$$ Furthermore, for $|v|=1$, $|v|Df_p(v)=1$. Thus, over the interval $(0,1)$ the mapping $p\mapsto \knorm{f_p}$ is bounded above. By (\ref{prodrule2}), for $v\in T^*$, we have
\ben &\,&\hskip-20pt|v|\prod_{j=0}^{k-1}\ell_j(|v|)D\psi(v)f_p(v)\le |v|\prod_{j=0}^{k-1}\ell_j(|v|)D(\psi f_p)(v)
 +|v|\prod_{j=0}^{k-1}\ell_j(|v|)|\psi(v^-)|D f_p(v)\nonumber\\
&\,&\hskip90pt\le \knorm{M_\psi f_p}+\|\psi\|_\infty\knorm{f_p}.\nonumber\eeqn
This implies that $$|v|\prod_{j=0}^{k-1}\ell_j(|v|)D\psi(v)(\ell_k(|v|))^p\le (\|M_\psi\|+\|\psi\|_\infty)\knorm{f_p}.$$
Letting $p$ approach 1, by the boundedness of $\knorm{f_p}$, we obtain $\psi\in \Lip^{(k+1)}$.

(c)$\Longrightarrow$(b):  Assume $\psi\in L^\infty\cap\Lip^{(k+1)}$ and let $f\in\Lipko$. By (\ref{prodrule}) and Proposition~\ref{lm0}, for $|v|>1$ we have
\small \ben |v|\prod_{j=0}^{k-1}\ell_j(|v|)D(\psi f)(v)&\le &|v|\prod_{j=0}^{k-1}\ell_j(|v|)D\psi(v)|f(v)|+|v|\prod_{j=0}^{k-1}\ell_j(|v|)|\psi(v^-)|Df(v)\nonumber\\
&\le &|v|\prod_{j=0}^{k}\ell_j(|v|)D\psi(v)\frac{|f(v)|}{\ell_k(|v|)}+\|\psi\|_\infty|v|\prod_{j=0}^{k-1}\ell_j(|v|)Df(v)\to 0\nonumber\eeqn \normalsize
as $|v|\to\infty$. Therefore, $\psi f\in\Lipko$. The boundedness of $M_\psi$ on $\Lipko$ follows from Lemma~\ref{drs}.

Finally, the lower estimate on the operator norm follows at once by applying Corollary \ref{funct_Banach} and by taking as test function the constant 1.\end{proof}

\section{Spectrum of $M_\psi$}\label{Section:Spectrum}
In this section, we study the spectra of the bounded multiplication operator $M_\psi$ on $\Lip^{(k)}$ and $\Lip_0^{(k)}$ and characterize those operators that are bounded below.

Denote by $\sigma(S)$, $\sigma_p(S)$, and $\sigma_{ap}(S)$ the spectrum, the point spectrum, and the approximate point spectrum, respectively, of an operator $S$ on a Banach space $X$.
Recall that $\s(S)$ is nonempty and closed, and the following inclusions hold \cite{Conway:07}:
\ben \s_p(S)\subseteq \s_{ap}(S)\subseteq\s(S).\label{spectra}\eeqn

\begin{proposition}[Proposition 6.7 of \cite{Conway:07}]\label{boundary spectrum} If $S$ is a bounded operator on a Banach space, then the boundary of $\s(S)$ is a subset of $\s_{ap}(S)$. \end{proposition}

\begin{theorem}\label{spectrum} Let $k \in \N$ and $M_\psi$ be a bounded multiplication operator on $\Lip^{(k)}$ or $\Lip_0^{(k)}$. Then
\begin{enumerate}
\item[\rm{(a)}]$\sigma_p(M_\psi) = \psi(T)$.
\item[\rm{(b)}]$\s(M_\psi) = \overline{\psi(T)}$.
\item[\rm{(c)}]$\s_{ap}(M_\psi) = \overline{\psi(T)}$.
\end{enumerate}
\end{theorem}

\begin{proof}  We will prove this result for $M_\psi$ bounded on $\Lip_0^{(k)}$.  The same argument holds for $M_\psi$ bounded on $\Lip^{(k)}$.

To prove (a), suppose $\lambda \in \sigma_p(M_\psi)$.  Then there exists a function $f \in \Lip_0^{(k)}$ that is not identically zero such that $M_{\psi-\lambda}f$ is identically zero.  Thus, there exists $v \in T$ such that $f(v) \neq 0$.  Then $0=(M_{\psi-\lambda}f)(v) = (\psi(v)-\lambda)f(v)$, and so $\psi(v) = \lambda$, proving that $\lambda$ is in the image of $\psi$.
Conversely, suppose $\lambda$ is in the image of $\psi$.  Then there exists $w \in T$ such that $\psi(w) = \lambda$.  So we see that $M_{\psi-\lambda}\chi_{w}$ is identically zero.  Thus $\lambda \in \sigma_p(M_\psi)$.  Therefore $\sigma_p(M_\psi)=\psi(T)$.

To prove (b), note that the inclusion $\overline{\psi(T)}\subseteq \s(M_\psi)$ follows at once from part (a) by passing to the closure.
Conversely, if $\l\notin \overline{\psi(T)}$, then there exists a positive constant $c$ such that $|\psi(v)-\l|\ge c$ for all $v\in T$. Thus, the function $g=(\psi-\l)^{-1}$ is bounded on $T$.  By the boundedness of $M_\psi$ on $\Lip_0^{(k)}$ and Theorem~\ref{boundedness}, it follows that $\psi\in L^\infty\cap \Lip^{(k+1)}$.  Therefore,
\ben \sup_{v\in T^*}|v|\prod_{j=1}^{k}\ell_j(|v|)Dg(v)&=&\sup_{v\in T^*}|v|\prod_{j=1}^{k}\ell_j(|v|)\left|\frac1{\psi(v)-\l}-\frac1{\psi(v^-)-\l}\right|\nonumber\\
&\le& \frac1{c^2}\sup_{v\in T^*}|v|\prod_{j=1}^{k}\ell_j(|v|)D\psi(v)<\infty.\nonumber\eeqn
Thus $g \in L^\infty\cap\Lip^{(k+1)}$, and so, by Theorem~\ref{boundedness}, the operator $M_g=M_{(\psi-\l)^{-1}}=M_{\psi-\l}^{-1}$ is bounded on $\Lip_0^{(k)}$. Hence $\l\notin \s(M_\psi)$. Therefore $\s(M_\psi)=\overline{\psi(T)}$.

To prove (c), we first note that by inclusion (\ref{spectra}) and part (b), we have $\s_{ap}(M_\psi) \subseteq \overline{\psi(T)}$.  Also, by inclusion (\ref{spectra}) and part (a), we have $\psi(T) \subseteq \s_{ap}(M_\psi)$.  Finally, by Proposition \ref{boundary spectrum}, the boundary of $\psi(T)$ is contained in $\s_{ap}(M_\psi)$.  Thus, we have $\overline{\psi(T)} \subseteq \s_{ap}(M_\psi)$, which means $\s_{ap}(M_\psi) = \overline{\psi(T)}$.
\end{proof}

A bounded operator $S$ on a Banach space $X$ is said to be \emph{bounded below} if there exists a positive constant $c$ such that $\|Sx\| \geq c\|x\|$ for all $x \in X$.  Note that a bounded operator that is bounded below is necessarily injective.

	The following result connects the approximate point spectrum and the operators that are bounded below.

\begin{proposition}[Proposition 6.4 of \cite{Conway:07}]\label{conway approximate point spectrum} If $S$ is a bounded operator on a Banach space, then $\lambda \not\in \sigma_{ap}(S)$ if and only if $S-\lambda I$ is bounded below.
\end{proposition}

	We next characterize the bounded multiplication operators on $\Lip^{(k)}$ or $\Lip_0^{(k)}$ which are bounded below.

\begin{theorem}\label{bounded below} If $M_\psi$ is a bounded multiplication operator on $\Lip^{(k)}$ or $\Lip_0^{(k)}$ for $k \in \N$, then $M_\psi$ is bounded below if and only if $\inf\limits_{v\in T}|\psi(v)|>0$.
\end{theorem}

\begin{proof} By Proposition~\ref{conway approximate point spectrum}, if $M_\psi$ is a bounded operator on $\Lip^{(k)}$ or $\Lip_0^{(k)}$, then $M_\psi$ is bounded below if and only if $0\notin \s_{ap}(M_\psi)$. By Theorem \ref{spectrum}, this condition is equivalent to $0\notin \overline{\psi(T)}$, i.e. $\inf\limits_{v\in T}|\psi(v)|>0$.\end{proof}

\section{Isometries and zero divisors}\label{isometries}
In this section we extend Theorem~9.1 in \cite{AllenColonnaEasley:10} to the iterated logarithmic Lipschitz spaces.

\begin{theorem}\label{iso_constant} The only isometric multiplication operators on $\Lip^{(k)}$ or $\Lip_0^{(k)}$ are induced by the constant functions of modulus one.
\end{theorem}

\begin{proof} It is clear that the constant functions of modulus one are symbols of isometric multiplication operators on $\Lip^{(k)}$ and $\Lip_0^{(k)}$. Thus, assume $M_\psi$ is an isometry on $\Lip^{(k)}$ or $\Lip_0^{(k)}$ so that, in particular,
\ben \|\psi\|_k= \|M_\psi 1\|_k = 1.\label{onenorm}\eeqn
First we are going to show that $\psi$ has constant modulus 1.  Fix $v\in T^*$ and let $f_v=\displaystyle(|v|+1)^{-1}\displaystyle\prod_{j=1}^{k-1}(\ell_j(|v|+1))^{-1}\c_v$. Then $|\psi(v)|=\|M_\psi f_v\|_k =\|f_v\|_k=  1$. On the other hand, for $g=\frac12\chi_o$, we have
$|\psi(o)|=\|M_\psi g\|_k=\|g\|_k=1.\,$
 Hence, $|\psi(v)|=1$ for all $v\in T$. From (\ref{onenorm}) it follows that $D\psi$ must be identically 0. Therefore, $\psi$ is a constant function of modulus one.\end{proof}

\begin{definition}\label{zero_divisor} Let $X$ be a Banach space of functions defined on a tree $T$ and let $Z$ be a nonempty subset of $T$. A function $\psi\in X$ is called a \emph{zero divisor} for $Z$ if it vanishes precisely at the vertices in $Z$ and $g/\psi\in X$ for every $g\in X$ vanishing on $Z$. The function $\psi$ is said to be an \emph{isometric zero divisor} if $\|g/\psi\|=\|g\|$ for each such function $g$.
\end{definition}

Recalling the set $\mathcal{P}$ in Proposition~\ref{dense}, it was shown in \cite{AllenColonnaEasley:10} that under certain hypotheses on the space $X$, the isometric zero divisors induce isometric multiplication operators on $X$.

\begin{theorem}[Theorem 9.3 of \cite{AllenColonnaEasley:10}]\label{zero_divisor_iso} Let $X$ be a functional Banach space on $T$ containing $\mathcal{P}$ and satisfying the following properties:
\begin{enumerate}
\item[\rm{(a)}] $\mathcal{P}$ is dense in $X$.
\item[\rm{(b)}] For each $v\in T$ and each $f\in X$, $p_vf\in X$, where $p_v=\chi_{S_v}$.
\end{enumerate}
If $\psi\in X$ is an isometric zero divisor, then $M_\psi$ is an isometry on $X$.\end{theorem}

\begin{lemma}\label{Lipko satisfies conditions} For every $k \in \N$, $\Lipko$ is a functional Banach space on $T$ containing $\mathcal{P}$ and satisfying the following properties:
\begin{enumerate}
\item[\rm{(a)}] $\mathcal{P}$ is dense in $\Lipko$.
\item[\rm{(b)}] For each $v\in T$ and each $f\in \Lipko$, $p_vf\in \Lipko$.
\end{enumerate}
\end{lemma}

\begin{proof}  By Corollary~\ref{funct_Banach}, $\Lipko$ is a functional Banach space, and by Proposition \ref{dense}, $\mathcal{P}$ is dense in $\Lipko$. Property (b) is immediate.
\end{proof}

We now turn our attention to the existence of isometric zero divisors on $\Lipko$.

\begin{corollary}\label{no_iso_zero} The space \rm{$\Lipko$} has no isometric zero divisors for each $k \in \N$.
\end{corollary}

\begin{proof}  Assume $\psi$ is an isometric zero divisor of $\Lipko$.  Thus there exists a non-empty set $Z$ on which $\psi$ vanishes.  Then by Lemma \ref{Lipko satisfies conditions} and Theorem \ref{zero_divisor_iso}, $M_\psi$ is an isometry on $\Lipko$. However, by Theorem~\ref{iso_constant}, $\psi$ is a constant modulus 1 and hence nonvanishing, a contradiction.
\end{proof}

\begin{theorem}\label{no_iso_zero_main} The space \rm{$\Lipk$} has no isometric zero divisors for each $k \in \N$.
\end{theorem}

\begin{proof} Suppose $\psi$ is an isometric zero divisor of $\Lipk$. In particular, for each $v\in T^*$, the function $\psi p_v$ vanishes at the points where $\psi$ vanishes, thus
$$\|p_v\|_k=\left\|\frac{\psi p_v}{\psi}\right\|_k=\|\psi p_v\|_k,\,\hbox{ for each }v\in T^*.$$ More generally, $\|f\|_k=\|\psi f\|_k$ for each $f\in\mathcal{P}$.

We begin by showing that $\psi$ is bounded.
Fix $w\in T^*$ and let $$f_w=\frac1{|w|\prod_{j=0}^{k-1}\ell_j(|w|)}p_w.$$ Then $f_w\in\mathcal{P}$ and, recalling that $Dp_w=\chi_w$, for $v\in T^*$,
$$|v|\prod_{j=0}^{k-1}\ell_j(|v|)Df_w(v)=\frac{|v|\prod_{j=0}^{k-1}\ell_j(|v|)}{|w|\prod_{j=0}^{k-1}\ell_j(|w|)}\chi_w(v)=\chi_w(v).$$
Thus, $\|\psi f_w\|_k=\|f_w\|_k=1$. Therefore, letting $D_w=S_w\backslash\{w\}$,
 we deduce
\ben 1&=&\sup_{v\in S_w}|v|\prod_{j=0}^{k-1}\ell_j(|v|)|\psi(v)f_w(v)-\psi(v^-)f_w(v^-)|\nonumber\\
&=&\max\left\{|\psi(w)|,\sup_{v\in D_w}\frac{|v|\prod_{j=0}^{k-1}\ell_j(|v|)}{|w|\prod_{j=0}^{k-1}\ell_j(|w|)}D\psi(v)\right\}\ge |\psi(w)|,
\nonumber\eeqn
proving that $\psi\in L^\infty$.

Next, for $|w|\ge 2$, let us define
$$g_w(v)=\begin{cases}\ \ \ \ \  0 &\quad\hbox{ if }v=o,\\
\frac{\ell_k(|v|)}{|w|\prod_{j=0}^{k-1}\ell_j(|w|)}&\quad\hbox{ if }1\le |v|<|w|,\\
\frac{\ell_k(|w|)}{|w|\prod_{j=0}^{k-1}\ell_j(|w|)}&\quad\hbox{ if }|v|\ge |w|.\end{cases}$$
Then $g_w\in\mathcal{P}$, so that
\ben \|\psi g_w\|_k=\|g_w\|_k.\label{fact0}\eeqn Furthermore, by Lemma~\ref{noteasy},
\ben \|g_w\|_k&=&\frac{\max\left\{1,\sup\limits_{2\le |v|\le |w|}|v|\prod_{j=0}^{k-1}\ell_j(|v|)(\ell_k(|v|)-\ell_k(|v|-1))\right\}}{|w|\prod_{j=0}^{k-1}\ell_j(|w|)}\nonumber\\
&=&\frac{\max\{1,2\prod_{j=0}^{k-1}\ell_j(2)(\ell_k(2)-\ell_k(1))\}}{|w|\prod_{j=0}^{k-1}\ell_j(|w|)}.\label{fact1}\eeqn
On the other hand
\ben \|\psi g_w\|_k\hskip-5pt&\ge &\hskip-8pt\sup_{2\le |v|\le |w|}\frac{|v|\prod_{j=0}^{k-1}\ell_j(|v|)}{|w|\prod_{j=0}^{k-1}\ell_j(|w|)}|D\psi(v)\ell_k(|v|)+\psi(v^-)(\ell_k(|v|)
-\ell_k(|v|-1))|\nonumber\\&\ge &\hskip-8pt\sup_{2\le |v|\le |w|}\frac{|v|\prod_{j=0}^{k-1}\ell_j(|v|)}{|w|\prod_{j=0}^{k-1}\ell_j(|w|)}D\psi(v)\ell_k(|v|)\nonumber\\&\ &\hskip10pt-\sup_{2\le|v|\le |w|}|\psi(v^-)|\frac{|v|\prod_{j=0}^{k-1}\ell_j(|v|)(\ell_k(|v|)-\ell_k(|v|-1))}{|w|\prod_{j=0}^{k-1}\ell_j(|w|)}\nonumber\\&\ge &\hskip-8pt\sup_{2\le |v|\le |w|}\frac{|v|\prod_{j=0}^{k-1}\ell_j(|v|)}{|w|\prod_{j=0}^{k-1}\ell_j(|w|)}D\psi(v)\ell_k(|v|)\nonumber\\&\ &\hskip10pt-\|\psi\|_\infty\frac{2\prod_{j=0}^{k-1}\ell_j(2)(\ell_k(2)-\ell_k(1))}{|w|\prod_{j=0}^{k-1}\ell_j(|w|)}.\label{fact2}\eeqn
Combining (\ref{fact0}), (\ref{fact1}) and (\ref{fact2}), we obtain
$$\sup_{2\le |v|\le |w|}\frac{|v|\displaystyle\prod_{j=0}^{k-1}\ell_j(|v|)}{|w|\displaystyle\prod_{j=0}^{k-1}\ell_j(|w|)}D\psi(v)\ell_k(|v|)\le \frac{1+(1+\|\psi\|_\infty)2\displaystyle\prod_{j=0}^{k-1}\ell_j(2)(\ell_k(2)-\ell_k(1))}{|w|\displaystyle\prod_{j=0}^{k-1}\ell_j(|w|)}.$$
Multiplying both sides by $|w|\prod_{j=0}^{k-1}\ell_j(|w|)$ and letting $|w|\to \infty$, we obtain
$$\sup_{|v|\ge 2}|v|\prod_{j=0}^k\ell_j(|v|)D\psi(v)<\infty.$$ Therefore, $\psi\in \Lip^{(k+1)}$. In particular, since $\ell_k(|v|)\to \infty$ as $|v|\to\infty$, it follows that
$$\lim_{|v|\to\infty}|v|\prod_{j=0}^{k-1}\ell_j(|v|)D\psi(v)=0.$$ Therefore, $\psi\in \Lipko$. The result now follows from Corollary~\ref{no_iso_zero}.
\end{proof}

\section{Compactness}\label{Section:Compact}

We begin this section, with a standard compactness criterion, which will be needed to prove one of our main results of the paper.

\begin{lemma}\label{compact:chara} A bounded multiplication operator $M_\psi$ on $\Lipk$ (respectively, $\Lipko$) is compact if and only if for every bounded sequence $\{f_n\}$ in $\Lipk$ (respectively, $\Lipko$) converging to 0 pointwise, $\knorm{\psi f_n}\to 0$ as $n\to\infty$.\end{lemma}

\begin{proof} We prove the result for the bounded operator $M_\psi$ acting on $\Lipk$. The proof for the other case is similar.

	Assume first $M_\psi$ is compact on $\Lipk$ and let $\{f_n\}$ be a bounded sequence in $\Lipk$ converging to 0 pointwise. By rescaling the sequence, if necessary, we may assume $\knorm{f_n} \le 1$ for all $n\in \N$. The compactness assumption implies that $\{f_n\}$ has a subsequence $\{f_{n_h}\}$ such that $\{\psi f_{n_h}\}$
 converges in norm to some function $f\in\Lipk$. Then $\psi(o) f_{n_h}(o)\to f(o)$ and by Proposition~\ref{modulus_est} applied to the function $\psi f_{n_h}-f$, for each $v\in T^*$, we have $|\psi(v)f_{n_h}(v)-f(v)|\le \ell_k(|v|)\knorm{\psi f_{n_h}-f}.$
Therefore, $\psi f_{n_k}\to f$ pointwise. Since by assumption, $f_n\to 0$ pointwise, the function $f$ must be identically 0, hence $\knorm{\psi f_{n_h}}\to 0$. Since $0$ is the only limit point in $\Lipk$ of the sequence $\{\psi f_n\}$, it follows that $\knorm{\psi f_n}\to 0$ as $n\to\infty$.

	Conversely, assume $\knorm{\psi f_n}\to 0$ as $n\to\infty$ for every bounded sequence $\{f_n\}$ in $\Lipk$ converging to 0 pointwise. Let $\{g_n\}$ be a sequence in $\Lipk$ with $\knorm{g_n}\le 1$. Then $|g_n(o)|\le 1$ and by Proposition~\ref{modulus_est}, $|g_n(v)|\le \ell_k(|v|)$ for each $v\in T^*$. Thus $\{g_n\}$ is uniformly bounded on finite subsets of $T$ and so some subsequence $\{g_{n_h}\}$ converges pointwise to some function $g$.

Fix $\e>0$ and $v\in T$, $|v|\ge 2$. Then $$|g_{n_h}(v)-g(v)|<\frac{\e}{2|v|\prod_{j=0}^{k-1}\ell_j(|v|)} \ \hbox{ and }\ |g_{n_h}(v^-)-g(v^-)|<\frac{\e}{2|v|\prod_{j=0}^{k-1}\ell_j(|v|)}$$ for all $h$ sufficiently large.
Thus
\ben |v|\prod_{j=0}^{k-1}\ell_j(|v|)Dg(v)&\le& |v|\prod_{j=0}^{k-1}\ell_j(|v|)|g(v)-g(v^-)-(g_{n_h}(v)-g_{n_h}(v^-))|\nonumber\\&\ &\quad +|v|\prod_{j=0}^{k-1}\ell_j(|v|)D g_{n_h}(v)<\e +\|g_{n_h}\|_k.\nonumber\eeqn
Therefore $g\in\Lipk$.
The sequence $\{f_h\}$ defined by $f_h=g_{n_h}-g$ is bounded in $\Lipk$ and converges to 0 pointwise. By the assumption, $\knorm{\psi f_{n_h}}\to 0$ as $h\to\infty$, so $\psi g_h\to\psi g$ in norm. Hence $M_\psi$ is compact.\end{proof}

	We are now ready to prove the main result of this section. Let $L_0$ denote the space of functions $f$ on $T$ vanishing at infinity (i.e. satisfying the condition $\lim_{|v|\to\infty}f(v)=0$).

\begin{theorem}~\label{chara_compactness} For a bounded multiplication operator $M_\psi$ on $\Lipk$, the following statements are equivalent.

{\rm{(a)}} $M_\psi$ is compact on $\Lipk$.

{\rm{(b)}} $M_\psi$ is compact on $\Lipko$.

{\rm{(c)}} $\psi\in L_0\cap\Lip^{(k+1)}_0$.
\end{theorem}

\begin{proof} (a)$\Longrightarrow$(b) is obvious.

(b)$\Longrightarrow$(c):  Assume $M_\psi$ is compact on $\Lipko$. To prove (c) it suffices to show that if $\{v_n\}$ be a sequence in $T$ such that $2< |v_n|\to\infty$, then \ben \displaystyle\lim_{n\to\infty}\psi(v_n)=0\ \hbox{ and}\label{f1}\eeqn \ben \lim\limits_{n\to \infty}|v_n|\prod_{j=0}^{k}\ell(|v_n|)D\psi(v_n)=0.\label{f2}\eeqn
Corresponding to each integer $n>2$, define
$\displaystyle f_n=\frac1{|v_n|\prod_{j=0}^{k-1}\ell_j(|v_n|)}\c_{v_n}.$ Then $f_n\in\Lipko$, $f_n\to 0$ pointwise, and an easy inductive argument shows that
$$\knorm{f_n}=\frac{|v_n|+1}{|v_n|}\prod_{j=0}^{k-1}\frac{\ell_j(|v_n|+1)}{\ell_j(|v_n|)}\le 2^k.$$ Using Lemma~\ref{compact:chara}, we obtain $|\psi(v_n)|\le \knorm{\psi f_n}\to 0$ as $n\to\infty$, proving (\ref{f1}).

Next, for $n\in\N$, let $$g_n(v)=\begin{cases} 0& \quad\hbox{ if }|v|=0,1\\
\frac{[\ell_k(|v|)]^2}{\ell_k(|v_n|)} & \quad\hbox{ if }2\le |v|<|v_n|-1,\\
\ell_k(|v_n|) & \quad\hbox{ if }|v|\ge |v_n|-1.\end{cases}$$
Then $g_n\in\Lipko$, $g_n(o)=0$, $\{g_n\}$ converges to 0 pointwise, and for $v\in T^*$
$$Dg_n(v)=\begin{cases} 0& \quad\hbox{ if }|v|=1 \hbox{ or }|v|>|v_n|-1,\\
\frac{[\ell_k(2)]^2}{\ell_k(|v_n|)} & \quad\hbox{ if }|v|=2,\\
\frac{[\ell_k(|v|)]^2-[\ell_k(|v|-1)]^2}{\ell_k(|v_n|)} & \quad\hbox{ if }2<|v|<|v_n|-1,\\
\frac{[\ell_k(|v_n|)]^2-[\ell_k(|v_n|-2)]^2}{\ell_k(|v_n|)} & \quad\hbox{ if }|v|=|v_n|-1.\end{cases}$$
Using the inequality $\ell_k(|v|)+\ell_k(|v|-1)\le 2\ell_k(|v_n|)$ for $2<|v|<|v_n|$, noting that   $$[\ell_k(|v_n|)]^2-[\ell_k(|v_n|-2)]^2\le 2\ell_k(|v_n|)(\ell_k(|v_n|)-\ell_k(|v_n|-1)+\ell_k(|v_n|-1)-\ell_k(|v_2|-2),$$ and applying Lemma~\ref{noteasy}, we obtain
$$|v|\prod_{j=0}^{k-1}\ell_j(|v|)Dg_n(v)\le \max\left\{2\prod_{j=0}^k\ell_j(2),12
\prod_{j=0}^{k-1}\ell_j(3)(\ell_k(3)-\ell_k(2))\right\}.$$ Thus, $\{\knorm{g_n}\}$ is bounded. By Lemma~\ref{compact:chara}, we obtain $$|v_n|\prod_{j=0}^{k}\ell_j(|v_n|)D\psi(v_n)\le \knorm{\psi g_n}\to 0$$ as $n\to\infty$, proving (\ref{f2}).

(c)$\Longrightarrow$(a): Set aside the case when $\psi$ is the constant 0. By Lemma~\ref{compact:chara}, to prove that $M_\psi$ is compact on $\Lipk$, it suffices to show that if $\{f_n\}$ is a sequence in $\Lipk$ converging to 0 pointwise and such that $a=\displaystyle\sup_{n\in\N}\knorm{f_n}<\infty$, then $\knorm{\psi f_n}\to 0$ as $n\to\infty$. Let $\{f_n\}$ be such a sequence and fix $\e>0$. Then $|f_n(o)|<\displaystyle\frac{\e}{3\knorm{\psi}}$ for all $n$ sufficiently large, and there exists $M\in \N$ such that for $|v|\ge M$, $|\psi(v)|<\e/(3a)$, and $|v|\prod_{j=0}^{k}\ell_j(|v|)D\psi(v)<\e/(3a).$
If $|v|> M$, then $|v^-|\ge M$, so $|\psi(v^-)|<\e/(3a)$. Thus, by (\ref{prodrule}) and Proposition~\ref{modulus_est}, we obtain
$$ |v|\prod_{j=0}^{k-1}\ell_j(|v|)D(\psi f_n)(v) \le |v|\prod_{j=0}^{k}\ell_j(|v|)D\psi(v)\knorm{f_n}+\frac{\e}3.$$
Since $f_n\to 0$ uniformly on $\{v\in T: |v|\le M\}$, so does the sequence $\{w\mapsto |w|\prod_{j=0}^{k-1}\ell_j(|w|)D(\psi f_n)(w)\}$. Thus, for $n$ sufficiently large and for each $v\in T^*$, $|v|\prod_{j=0}^{k-1}\ell_j(|v|)|D(\psi f_n)(v)<\frac23\e$. Therefore $\knorm{\psi f_n}<\e$ for all $n$ sufficiently large, whence $\knorm{\psi f_n}\to 0$, as $n\to\infty$.\end{proof}

\section{Essential norm}\label{ess_norm}

Recall that the \emph{essential norm} of a bounded operator $S$ on a Banach space $X$ is defined as $$\|S\|_e=\inf\{\|S-K\|: K \hbox{ compact operator on }X\}.$$
\vskip3pt

 Given a bounded multiplication operator $M_\psi$ on $\Lipk$ or $\Lipko$, define
\ben A(\psi)=\lim_{n\to\infty}\sup_{|v|\ge\, n}|\psi(v)|\,\hbox{ and }\,
B(\psi)=\lim_{n\to\infty}\sup_{|v|\ge\, n}|v|\prod_{j=0}^{k}\ell_j(|v|)D\psi(v).\nonumber\eeqn

\begin{theorem}\label{estimate} Let $M_\psi$ be bounded on \rm{$\Lipk$} or \rm{$\Lipko$}. Then
$$\max\left\{A(\psi),B(\psi)\right\}\le \|M_\psi\|_e\le A(\psi)+B(\psi).$$
\end{theorem}

\begin{proof} We begin with the proof of the lower estimate. For each $n\in\N$ and $v\in T$, define
$f_n=\frac1{d_n}\c_{\{v:\,|v|=n\}},$ where $d_n=n\prod_{j=0}^{k-1}\ell_j(n).$ Then $f_n\in\Lipko$, $\|f_n\|_k =\displaystyle\frac{(n+1)\prod_{j=0}^{k-1}\ell_j(n+1)}{n\prod_{j=0}^{k-1}\ell_j(n)}\le 2^k$, and $f_n\to 0$ pointwise. Therefore, by Proposition~\ref{weakconv}, the sequence $\{f_n\}$ converges weakly to 0 in $\Lipko$ and hence  $\displaystyle\lim\limits_{n\to\infty}\|Kf_n\|_k=0$ for any compact operator $K$ on $\Lipko$. Consequently,
$$\|M_\psi-K\|\ge \limsup_{n\to\infty}\|M_\psi f_n\|_k.$$
By Lemma~\ref{easy}, we deduce that
\ben \|M_\psi\|_e&\ge & \limsup_{n\to\infty} \|M_\psi f_n\|_k\nonumber\\
&=& \limsup_{n\to\infty} \sup_{v\in T^*}|v|\prod_{j=0}^{k-1}\ell_j(|v|)|\psi(v)f_n(v)-\psi(v^-)f_n(v^-)|\nonumber\\
&= &\lim_{n\to\infty} \frac{n+1}{n}\prod_{j=0}^{k-1}\frac{\ell_j(n+1)}{\ell_j(n)}\sup_{|v|\ge n}|\psi(v)|=A(\psi).\nonumber\eeqn

 We next show that $\|M_\psi\|_e\ge B(\psi).$
The result is immediate if $B(\psi)=0$. So assume $\{v_n\}$ is a sequence of vertices of length greater than 1 such that $|v_n|$ is increasing unboundedly and $$\lim_{n\to\infty}|v_n|\prod_{j=0}^k\ell_j(|v_n|)D\psi(v_n)=B(\psi).$$
Let $p$ be a fixed number in $(0,1)$. For each $n\in \N$, define
$$ h_n(v)=\begin{cases}0 &\hbox{ if }\quad |v|=0,\\
\displaystyle\frac{(\ell_k(|v|+1))^{p+1}}{(\ell_k(|v_n|))^p} &\hbox{ if }\quad 1\le |v|< |v_n|,\\
\ell_k(|v_n|)  &\hbox{ if }\quad |v|\ge |v_n|.\end{cases}$$
Then $h_n(o)=0$, $h_n(v_n)=h_n(v_n^-)=\ell_k(|v_n|)$, and
$$|v|\prod_{j=0}^{k-1}\ell_j(|v|)D h_n(v)=\hskip-3pt\begin{cases} \frac{(\ell_k(2))^{p+1}}{(\ell_k(|v_n|))^p}&\hbox{ if }\ |v|=1,\\
|v|\prod_{j=0}^{k-1}\ell_j(|v|)\frac{(\ell_k(|v|+1))^{p+1}-(\ell_k(|v|))^{p+1}}{(\ell_k(|v_n|))^p} &\hbox{ if }\ 2\le |v|< |v_n|,\\
\ \ 0 &\hbox{ if }  \ |v|\ge |v_n|.\end{cases}$$
For $2\le |v|<|v_n|-1$, calling $\g(|v|)$ the left-hand side of the above display, we see that
$$\frac{\g(|v|+1)}{\g(|v|)}=\frac{|v|+1}{|v|}\prod_{j=0}^{k-1}\frac{\ell_j(|v|+1)}{\ell_j(|v|)}\frac{\left(\frac{\ell_k(|v|+1)}{\ell_k(|v|)}\right)^{p+1}-1}{1-\left(\frac{\ell_k(|v|-1)}{\ell_k(|v|)}\right)^{p+1}}>0.$$ Thus, for $n$ sufficiently large, the supremum of $|v|\prod_{j=0}^{k-1}\ell_j(|v|)D h_n(v)$ is attained at the vertices of length $|v_n|-1$. Therefore, for $n$ sufficiently large, we have
\ben \|h_n\|_k=(|v_n|-1)\prod_{j=0}^{k-1}\ell_j(|v_n|-1)\frac{(\ell_k(|v_n|))^{p+1}-(\ell_k(|v_n|-1))^{p+1}}{(\ell_k(|v_n|))^p}.\nonumber\eeqn
Using Lemmas~\ref{easy} and \ref{noteasy}, we obtain
\ben \lim_{n\to\infty}\|h_n\|_k&=&\lim_{n\to\infty}|v_n|\prod_{j=0}^{k-1}\ell_j(|v_n|)(\ell_k(|v_n|)-\ell_k(|v_n|-1))\left(\frac{|v_n|-1}{|v_n|}\right)\nonumber\\
&\ &\hskip15pt \times \prod_{j=0}^{k-1}\frac{\ell_j(|v_n|-1)}{\ell_j(|v_n|)}\frac{(\ell_k(|v_n|))^{p+1}-(\ell_k(|v_n|-1))^{p+1}}{(\ell_k(|v_n|)-\ell_k(|v_n|-1))(\ell_k(|v_n|))^p},\nonumber\\
&=&\lim_{n\to\infty}\frac{1-u_n^{p+1}}{1-u_n}
=1+p,\nonumber\eeqn
where $u_n=\ell_k(|v_n|-1)/\ell_k(|v_n|)$.
 Then the function $g_n=\displaystyle\frac{h_n}{\|h_n\|_k}$ is in $\Lipko$, $\|g_n\|_k=1$, and $g_n\to 0$ pointwise. By Proposition~\ref{weakconv}, it follows that $\{g_n\}$ converges to $0$ weakly. Thus $\|Kg_n\|_k\to 0$ as $n\to\infty$ for any compact operator $K$ on $\Lipko$. Hence
$$\|M_\psi-K\|\ge \limsup_{n\to\infty}\|(M_\psi-K)g_n\|_k\ge \limsup_{n\to\infty}\|\psi g_n\|_k.$$ Taking the infimum over all such operators $K$, we obtain
\ben \|M_\psi\|_e&\ge &\limsup_{n\to\infty}|v_n|\prod_{j=0}^{k-1}\ell_j(|v_n|)|\psi(v_n)g_n(v_n)-\psi(v_n^-)g_n(v_n^-)|\nonumber\\
&=&\lim_{n\to\infty}|v_n|\prod_{j=0}^{k-1}\ell_j(|v_n)|)\frac{\ell_k(|v_n|)}{\|h_n\|_k}D\psi(v_n)
=\frac1{1+p}B(\psi).\nonumber\eeqn
Passing to the limit as $p\to 0$, we obtain $\|M_\psi\|_e\ge B(\psi).$

To prove the upper estimate, fix $n\in\N$, and define the operator $K_n$ on $\Lipk$ by
$$K_nf(v)=\begin{cases} f(v) &\hbox{ if }\quad |v|\le n,\\
f(v_n) &\hbox{ if }\quad |v|> n,\end{cases}$$
where $f\in\Lipk$ and $v_n$ is the ancestor of $v$ of length $n$. In particular, \ben K_nf(o)=f(o), \label{zerocondition}\eeqn
$K_nf$ attains finitely many values and hence is in $\Lipko$.

If $\{g_m\}$ is a sequence in $\Lipk$ with $\|g_m\|_k\le 1$ for each $m\in\N$, then, $a=\displaystyle\sup_{m\in\N}|g_m(o)|\le 1$ so that $|K_ng_m(o)|\le a$. Furthermore, by Proposition~\ref{modulus_est}, $|K_ng_m(v)|\le \ell_k(n)$ for each $v\in T^*$, and for each $m\in\N$. Therefore, some subsequence $\{K_ng_{m_i}\}_{i\in\N}$ must converge to a function $g$ on $T$ attaining constant values on the sectors determined by the vertices on the sphere centered at $o$ of radius $n$. In particular, $g\in\Lipk$ and, since $K_ng_{m_i}\to g$ uniformly on the closed ball centered at $o$ of radius $n$, and $DK_ng_{m_i}$ and $Dg$ are 0 outside of the ball, we get
\ben &\ &\|K_ng_{m_i}-g\|_k=|g_{m_i}(o)-g(o)|+\sup_{|v|\le n}|v|\prod_{j=0}^{k-1}\ell_j(|v|)D(g_{m_i}-g)(v)\nonumber\\
&\,&\hskip12pt \le |g_{m_i}(o)-g(o)|
+n\prod_{j=0}^{k-1}\ell_j(n)\sup_{|v|\le n}[|g_{m_i}(v)-g(v)|+\ |g_{m_i}(v^-)-g(v^-)|],\nonumber\eeqn
which converges to 0 as $i\to\infty$. This shows that $K_n$ is compact. Then the operator $M_\psi K_n$ is also compact.
For $v\in T^*$, we have
\ben |v|\prod_{j=0}^{k-1}\ell_j(|v|)D[(I-K_n)f](v)\le |v|\prod_{j=0}^{k-1}\ell_j(|v|)Df(v) \le \|f\|_k.\label{goodest}\eeqn
 Furthermore, by Proposition~\ref{modulus_est} and (\ref{goodest}), we see that
\ben |[(I-K_n)f](v)|\le \ell_k(|v|)\|(I-K_n)f\|_k\le \ell_k(|v|)\|f\|_k.\label{best}\eeqn
From (\ref{best}) and again from (\ref{goodest}), we obtain the estimate
\small \ben \|\psi(I-K_n)f\|_k&= &\sup_{|v|>n}|v|\prod_{j=0}^{k-1}\ell_j(|v|)|\psi(v)[(I-K_n)f](v)-\psi(v^-)[(I-K_n)f](v^-)|\nonumber\\
 &\le & \sup_{|v|>n}\{|\psi(v^-)||v|\prod_{j=0}^{k-1}\ell_j(|v|)D[(I-K_n)f](v) \nonumber\\
&\ & +\ |v|\prod_{j=0}^{k-1}\ell_j(|v|)D\psi(v)|[(I-K_n)f](v)|\}\nonumber\\
&\le &\sup_{|v|>n}|\psi(v^-)|\|f\|_k+\sup_{|v|>n}|v|\prod_{j=0}^{k}\ell_j(|v|)D\psi(v)\|f\|_k.\nonumber\eeqn
\normalsize
Consequently,
\ben \|M_\psi\|_e\le \limsup_{n\to\infty}\|M_\psi-M_\psi K_n\|
=\limsup_{n\to\infty}\sup_{\|f\|_k = 1}\|\psi(I-K_n)f\|_k
\le  A(\psi)+B(\psi),\nonumber\eeqn
completing the proof.\end{proof}

\section*{Acknowledgements}
The research of the first author is supported by a grant from the College of Science and Health of the University of Wisconsin-La Crosse.

\bibliographystyle{amsplain}
\bibliography{references.bib}
\end{document}